\documentclass[pdflatex,sn-mathphys]{sn-jnl}

\newcommand{\N}{{\mathbb N}}

\newcommand{\argmin}{\mathop{\mathrm{argmin}}}

\usepackage{hyperref}
\usepackage{bbold}
\usepackage{amsmath} 
\usepackage[cal=boondox]{mathalfa}
\usepackage{mathtools}
\DeclarePairedDelimiter\abs{\lvert}{\rvert}
\usepackage{ulem}

\usepackage[makeroom]{cancel}
\usepackage{lipsum}
\usepackage{amsfonts}
\usepackage{graphicx}
\usepackage{epstopdf}
\usepackage{mathrsfs}
\usepackage[caption=false]{subfig}
\usepackage{bm}

\jyear{2021}%

\theoremstyle{thmstyleone}%
\newtheorem{theorem}{Theorem}
\newtheorem{corollary}{Corollary}
\newtheorem{proposition}[theorem]{Proposition}%

\theoremstyle{thmstyletwo}%
\newtheorem{remark}{Remark}%
\newtheorem{hypothesis}{Hypothesis}

\theoremstyle{thmstylethree}%
\newtheorem{definition}{Definition}%

\begin{document}

\title[HPO: theoretical results]{Theoretical aspects in penalty hyperparameters optimization}

\author*[1,2]{\fnm{Flavia} \sur{Esposito}}\email{flavia.esposito@uniba.it}
\equalcont{These authors contributed equally to this work.}

\author[1,2]{\fnm{Laura} \sur{Selicato}}\email{laura.selicato@uniba.it}
\equalcont{These authors contributed equally to this work.}

\author[1,3]{\fnm{Caterina} \sur{Sportelli}}\email{caterina.sportelli@uniba.it}

\equalcont{These authors contributed equally to this work.}

\affil*[1]{\orgdiv{Department of Mathematics, Università degli Studi di Bari Aldo Moro}, 
\orgaddress{\street{via Orabona, 4}, \city{Bari}, \postcode{70125}, 
\country{Italy}}}

\affil[2]{\orgname{Member of INDAM research group, GNCS}}
\affil[3]{\orgname{Member of INDAM research group, GNAMPA}}

\abstract{Learning processes are useful methodologies able to improve knowledge of real phenomena. These are often dependent on hyperparameters, variables set before the training process and regulating the learning procedure. Hyperparameters optimization problem is an open issue in learning approaches since it can strongly affect any real data analysis. They are usually selected using Grid-Search or Cross Validation techniques. No automatic tuning procedure exists especially if we focus on an unsupervised learning scenario.\\
This study aims to assess some theoretical considerations for tuning penalty hyperparameters in optimization problems. It considers a bi-level formulation tuning problem in an unsupervised context, by using Gradient-based methods. Suitable conditions for the existence of a minimizer in an infinite-dimensional Hilbert space are outlined, together with some theoretical results, applicable in all those situations when it is unnecessary or not possible obtaining an exact minimizer. An iterative algorithmic strategy is considered, equipped with a stopping criterion via Ekeland's variational principle.}




\maketitle

\section*{Acknowledgment} The authors would like to thank Prof. A. M. Candela and Prof. N. Del Buono from Università degli Studi di Bari Aldo Moro for the deep discussions on the preliminary version of this manuscript. This work was supported by INDAM-GNCS.

\section{Introduction}
Training a Machine Learning (ML) algorithm is quite important to produce data-driven models, which can be successfully applied in real life applications. These processes often require to specify several variables by the users, namely hyperparameters, which must be set before the learning procedure starts. 
Hyperparameters govern the whole learning process and play a crucial role in guaranteeing good model performances. They are often manually specified, and the lack of an automatic tuning procedure makes the field of Hyperparameter Optimization (HPO) an ever-evolving topic. 
The literature offers various solutions for hyperparameters tuning, from Gradient-based to Black-Box or Bayesian's approaches, beside some naive but daily used methods such as Grid and Random search. A brief overview on existing methods can be found in~\cite{nostro}.
Hyperparameters can be of different types (discrete, continuous, categorical), and in most cases, the number of their configurations to explore is infinite. This paves the way for a mathematical formalization of the HPO in ML context with abstract spaces, such as Hilbert spaces.

A supervised learning algorithm may be represented as a mapping that takes a configuration of hyperparameter and a dataset $D$ and returns an hypothesis \cite{franceschi2021unified}:
\begin{equation}
    \mathcal{A} : \Lambda \times  D \to \mathcal{H}; \qquad \mathcal{A}(\lambda, D) = h, \quad \text{with} \quad D= \bigcup\limits_{N \in \mathbb{N}} (X \times Y)^N
\end{equation}
where $D$ is the space of finite dimensional dataset, representing a task, $X$ and $Y$ are the input and output spaces, $\Lambda$ is an hyperparameter space, and $\mathcal{H}$ is an hypothesis space. 
A quite standard claim for the hypotheses set is to be a linear function space, endowed with a suitable norm (more binding arising from an inner product): two requirements satisfied when $\mathcal{H}$ is a Hilbert space of functions over $X$\footnote{If $X$ is an infinite dimensional space the boundedness is needed, too.}. Assuming an Hilbert space structure on the hypothesis space has some advantages: (i) practical computations reduced to ordinary linear algebra operations and (ii) self duality; that is for any $x\in X$ a representative of $x$ can be found, i.e., $\mathcal{k}_x\in \mathcal{H}$ exists such that
\begin{equation} \label{scalarp}
h(x) = \langle \mathcal{k}_x, h\rangle\quad\mbox{ for all } h\in\mathcal{H}, 
\end{equation}
where $\mathcal{k}_x$ is a suitable positive definite ``kernel''. This construction gives the chance to connect the abstract structure of $\mathcal{H}$ and what its elements actually are, flipping the construction of the hypotheses set from the kernel. Providing a suitable positive function $k$ on $X$, $\mathcal{H}$ can be set as the minimal complete space of functions involving all $\{k_x\}_{x\in X}$ equipped with the scalar product in \eqref{scalarp}. Thus, $\mathcal{H}$ is outlined in a unique way, and it is named the Representing Kernel Hilbert Space mapped to the kernel $k$.

Starting from this abstract 
scenario, one can deepen the HPO in supervised ML. Formally, HPO can be formulated as the problem of minimizing the discrepancy between $\mathcal{A}$, trained on a given training dataset $D_{tr}$, and a validation dataset $D_{val}$~\cite{feurer2019hyperparameter}, to find the optimal $\lambda^*$ such that


\begin{equation} \label{eq1}
\lambda^*=\argmin_{\lambda \in \Lambda}
\mathscr{V}({\mathcal{A}}({\lambda}, D_{tr}), D_{val}), \quad \text{where} \quad  \mathscr{V}: \mathcal{H} \times X \to \mathbb{R}.
\end{equation}
\\

In this study, we will address problem \eqref{eq1} through Gradient-based methods (GB), by using a bi-level approach. Bi-level programming solves an outer optimization problem subject to the optimality of an inner optimization problem, and it can be adopted to formalize HPO for any learning algorithm~\cite{bard2013practical, colson2007overview,bertrand2020implicit,phdfranceschi}. We will work on Hilbert spaces for solving HPO in unsupervised problems,  considering as hyperparameter the penalty coefficient. We already treat this aspect in the particular and more specific case of Nonnegative Matrix Factorization task, encurring in some generalization problems and restrictions of the theorems' assumptions~\cite{nostro_Z}.

To overcome the difficulties in ensuring the theoretical assumptions when real data domains are considered, this work extends existence and uniqueness theorems for the solution of the hyperparameters bi-level problem to the more general framework of infinite dimensional Hilbert space. This latter also allows the application of the Ekeland's variational principle to state that whenever a functional is not guaranteed to have a minimum, under suitable assumptions, a “good” substitute can be found, namely the best one can get as an approximate minimum.
One of the purposes of this paper is to use this theoretical tool as a stopping criterion for the update of the hyperparameters as we will see later. 




The outline of the paper is as follows. Section \ref{sec:back} introduces the classical bi-level formalization of HPO and some preliminary notions in a supervised context. Section \ref{sec:prop} illustrates our proposal, extension on the unsupervised context. 
A general framework addressing HPO in Hilbert space is also set, and some general abstract tools are stated in Section \ref{sec:main}.  
Section \ref{sec:discussion} presents a critical discussion and some practical considerations. Finally, Section \ref{sec:conclusions} summarizes the obtained results and draws some conclusions.

\section{Previous works and preliminaries}\label{sec:back}
As briefly mentioned in the introduction, in a supervised learning scenario, HPO can be addressed through a bi-level formulation. This approach looks for the hyperparameters $\lambda$ such that the minimization of the regularized training leads to the best performance of the trained data-driven model on a validation set.   
Accordingly to the ideas introduced in~\cite{franceschi2017forward,vincent2018online}, the best hyperparameters for a data learning task can be selected as the solution of the following problem: 
\begin{align}\label{two}
\min\{J(\lambda): \lambda \in \Lambda\},\\
  J(\lambda) = \inf \{\mathcal{E}(w_{\lambda}, \lambda): w_{\lambda} \in \argmin\limits_{u \in \mathbb{R}^{r}}\mathscr{L}_{\lambda}(u)\},\label{twoo}
\end{align}
where $w \in \mathbb{R} ^r$ are $r$ parameters, $J: \Lambda \to \mathbb{R} $ is the so-called $Response$ $Function$ of the outer problem $\mathcal{E}:\mathbb{R} ^r \times \Lambda \to \mathbb{R}$, and for every $\lambda \in \Lambda \subset \mathbb{R}^p$, $\mathcal{L}_\lambda:\mathbb{R} ^r \to \mathbb{R} $ is the inner problem.
\\
A reformulation of HPO as a bi-level optimization problem 
is also solved via some GB algorithms.
In particular, in GB methods HPO is addressed with classical procedure for continuous optimization, in which the hyperparameter update is given by

\begin{equation}\label{lambdaupdate}
    \lambda_{t+1} = \lambda_t - \alpha \mathbf{h}_t(\lambda)
\end{equation}

where $\mathbf{h}_t$ is 
an approximation of the gradient of function $J$ and $\alpha$ is a step size. It is known that the main challenge in this context is the computation of $\mathbf{h}_t$, called hypergradient. In several cases, a numerical approximation of the hypergradient can be calculated for real-valued hyperparameters, although few learning algorithms are differentiable in the classical sense.\\
There are two main strategies for computing the hypergradient: iterative differentiation \cite{franceschi2017forward, fra2, Mc} and implicit differentiation \cite{Pedro,lorraine2020optimizing}.
The former requires calculating the exact gradient of an approximate objective. This is defined through the recursive application of an optimization dynamics that aims to replace and approximate the learning algorithm $\mathcal{A}$; the latter involves the numerical application of the implicit function theorem to the solution mapping $\mathcal{A}(D_{tr}; \cdot)$, when it is expressible through an appropriate equation \cite{franceschi2021unified}.

In this study, we follow the iterative strategy, so that problem in \eqref{two}-\eqref{twoo} can be addressed through a dynamical system type approach.

If the following hypothesis hold:
\begin{hypothesis}\label{Hyphot}\quad
\begin{enumerate}
    \item the set $\Lambda$ is a compact subset of $\mathbb{R}$;
    \item the Error Function $\mathcal{E} : \mathbb{R}^{r} \times \Lambda \to \mathbb{R}$ is jointly continuous;
    \item the map $(w, \lambda) \to \mathscr{L}_{\lambda}(w)$ is jointly continuous, and then the problem \\$\argmin \mathscr{L}_{\lambda}$ is a singleton for every $\lambda \in \Lambda$;
    \item the problem $w_{\lambda} = \argmin \mathscr{L}_{\lambda}$ remains bounded as $\lambda$ varies in $\Lambda$;
\end{enumerate}
\end{hypothesis}
the problem in \eqref{two}-\eqref{twoo} becomes:
\begin{equation}\label{tre}
    \min\limits_{\lambda \in \Lambda} J(\lambda) = \mathcal{E}(w^*_{\lambda}, \lambda), \quad w^*_{\lambda} = \argmin\limits_u \mathcal{L}_{\lambda}(u).
\end{equation}
It can be proved that the optimal solution $(w_{\lambda^*},\lambda^*)$ of problem \eqref{tre} exists~\cite{fra2}.\\
The goal of HPO is to minimize the validation error of model $g_{w}: X \to Y$, parameterized by a vector $w \in \mathbb{R} ^r$,  with respect to  hyperparameters $\lambda$.\\
Considering the penalty optimization problems in which hyperparameter is the penalty coefficient $\lambda\in\mathbb{R}_+$, the $Inner$ $Problem$ is the penalized empirical error represented by $\mathcal{L}$, defined as:
\begin{equation}\label{ell}
    \mathcal{L}_{\lambda} (w) = \sum\limits_{(x,y) \in D_{tr}} \ell (g_w (x),y) +  \lambda \mathcal{r}(w),
\end{equation}
where $\ell$ is the loss function, $D_{tr}=\{(\mathbf{x}_i,y_i)\}_{i=1}^n$ the training set, and $\mathcal{r}: \mathbb{R}^r \to \mathbb{R}$ is a penalty function.
While the $Outer$ $Problem$ is the generalized error of $g_w$ represented by $\mathcal{E}$: 
\begin{equation}
    \mathcal{E}(w,\lambda) = \sum\limits_{(x,y) \in D_{val}} \ell (g_w (x),y), 
\end{equation}
where $D_{val}=\{(\mathbf{x}_i,y_i)\}_{i=1}^n$ is the validation set. Note that $\mathcal{E}$ does not explicitly depend on $\lambda$.

This work will allow overcoming some assumptions of Hypothesis \ref{Hyphot} (such as compactness) that are difficult to satisfy in real data learning contexts, and also to use some theoretical result as the Ekeland's variational principle, stated in the following, to improve iterative algorithms.

\begin{theorem}[Ekeland's variational principle] \cite{ekeland1974variational}   \label{Eke}
Let $(V, d)$ be a complete metric space and $J:V\to\bar{\mathbb{R} }$ be a lower semi-continuous function which is bounded from below. Suppose that $\varepsilon>0$ and $\tilde{v}\in V$ exist such that
\[
J(\tilde{v})\leq \inf_V J +\varepsilon.
\]
Then, given any $\rho>0$, $v_{\rho}\in V$ exists such that
\[
J(v_{\rho})\leq J(\tilde{v}),\qquad d(v_{\rho}, \tilde{v})\leq\frac{\varepsilon}{\rho},
\]
and
\[
J(v_{\rho})<J(v)+\rho\, d(v_{\rho}, v) \qquad\forall\; v\neq v_{\rho}.
\]
\end{theorem} 

\section{Our Proposal}\label{sec:prop}

The bi-level HPO framework can be modified to include unsupervised learning paradigms, generally designed to detect some useful latent structure embedded in data. Tuning hyperparameters for unsupervised learning models is more complex than the supervised case due to the lack of the output space, which defines the ground truth collected in the validation set.    

This section describes a general framework to address HPO in Hilbert spaces for the unsupervised case and a corollary of the Ekeland's variational principle used to derive a useful stopping criterion for iterative algorithms solving this HPO. \\
Let $X \in \mathbb{R} ^{n \times m}$ be a data matrix, with reference to the problem \eqref{two}-\eqref{twoo}, where now $J: \Lambda \to \mathbb{R} $ is a suitable functional and $\Lambda$ a Hilbert space equipped with the scalar product $(\cdot, \cdot)$, the outer problem is:
\begin{equation}\label{E}
  \mathcal{E}:\mathbb{R} ^r \times \Lambda \to \mathbb{R}  \qquad  \mathcal{E}(w,\lambda) = \sum\limits_{x \in X} \ell (g_w (x)), 
\end{equation}
and for every $\lambda \in \Lambda$ the inner problem is:
\begin{equation}\label{L}
    \mathcal{L}:\mathbb{R} ^r \to \mathbb{R}  \qquad \mathcal{L}_{\lambda} (w) = \sum\limits_{x \in X} \ell (g_w (x)) + \mathcal{R}(\lambda, w),
\end{equation}
where $\mathcal{R}: \Lambda \times \mathbb{R}^r \to \mathbb{R}$ is a penalty function.
We want to emphasize the new formulation with respect to \eqref{ell} 
regarding the function $\mathcal{L}_\lambda$, in which each component of the parameter $w$ is penalized independently, and all optimization is performed on the data matrix $X$. 

The bi-level problem associated to \eqref{E}-\eqref{L} can be solved with a dynamical system approach in which the hypergradient is computed. Once the hypergradient is achieved a gradient-based approach can be used to find the optimum $\lambda^*$.
The Ekeland's variational principle can be used to construct an appropriate stopping criterion for iterative algorithms, with the aim of justifying and setting the hyperparameters related to the stopping criterion more appropriately.
Roughly speaking, this variational principle asserts that, under assumptions of lower semi-continuity and boundedness from below, if a point $\tilde{\lambda}$ is an “almost minimum point” for a function $J$, hence a small perturbation of $J$ exists which attains its minimum at a point “near” to $\tilde{\lambda}$. As a fruitful selection of $\rho$ occurs when $\rho =\sqrt{\varepsilon}$ and such a choice allows us to reduce the number of hyperparameters to the precision error only, thus we will use Theorem \ref{Eke} in the following form. 

\begin{corollary} \label{coroll}
Let $(V, d)$ be a complete metric space and 
$J: \Lambda \to\bar{\mathbb{R} }$ be a lower semi-continuous function which is bounded from below. Suppose that $\varepsilon>0$ and $\tilde{\lambda}\in \Lambda$ exist such that
\[
J(\tilde{\lambda})\leq \inf_{\Lambda} J +\varepsilon.
\]
Then, $\tilde{z}\in \Lambda$ exists such that
\[
J(\tilde{z})\leq J(\tilde{\lambda}),\qquad d(\tilde{z}, \tilde{\lambda})\leq\sqrt{\varepsilon}.
\]
and
\[
J(\tilde{z})<J(\lambda)+\sqrt{\varepsilon}\, d(\tilde{z}, \lambda) \quad\forall\; \lambda \neq \tilde{z}.
\]
\end{corollary}


\section{Main Abstract results}
\label{sec:main}
In this section, we are ready to weaken the assumptions we discussed earlier and provide results related to the use of the Ekeland\textcolor{blue}{'s} principle as a stopping criterion. 
We mention an abstract result of the existence of a minimizer in Hilbert spaces which has great importance and a wide range of applications in several fields.
As just one example, the Riesz's Representation Theorem, even if implicitly, makes use of the existence of a minimizer~\cite{walter1987real}. This is a widely relevant issue about Hilbert spaces, which makes them nicer than Banach spaces or other topological vector spaces. One can think, for example, that the whole Dirac Bra--ket formalism of quantum mechanics relies on this identification. 
\subsection{Abstract Existence Theorem}
It is well known that each bounded sequence in a normed space $\Lambda$ has a norm convergent subsequence if and only if it is a finite dimensional normed space. 
Thus, given a normed space $\Lambda$, as the strong topology (i.e., the one induced by the norm) is too strong to provide any widely appropriate subsequential extraction procedure, one can consider other weak topologies joined with the linear structure of the space and look for subsequential extraction processes therein.\\
In Banach spaces, as well as in Hilbert spaces, the two most relevant weaker-than-norm topologies are the weak-star topology and the weak topology. As the former is established in dual spaces, the latter is set up in every normed space. The notions of these topologies are not self-contained but fulfill a leading role in many features of the Banach space theory. In this regard, here we state some results we will use shortly.
\begin{theorem} \label{noto}
If $\Lambda$ is a finite-dimensional space, the strong and weak topologies coincide. In particular, it follows that the weak topology is normable, and then clearly metrizable, too.\\
If $\Lambda$ is an infinite-dimensional space, the weak topology is strictly
contained in the strong topology, namely open sets for the strong
topology exist which are not open for the weak topology. Furthermore, the weak
topology turns to be not metrizable in this case. 
\end{theorem}
\begin{definition}
A functional $J:\Lambda\to\bar{\mathbb{R} }$ with $\Lambda$ topological space, is said to be lower semi-continuous on $\Lambda$ if for each $a\in\mathbb{R} $, the sublevel sets
\[
J^{-1}(]-\infty, a]) =\{\lambda\in \Lambda: J(\lambda)\le a\}
\]
are closed subsets of $\Lambda$. 
\end{definition}
 In the following we introduce a ``generalized Weierstrass Theorem" which gives a criteria for the existence of a minimum for a functional defined on a Hilbert space. For this reason, the incoming results will be provided for the abstract framework of a Hilbert space although, in some cases, they apply in the more general context of Banach spaces. Thus, throughout the remaining part of this section we denote by $\Lambda$ any real infinite dimensional Hilbert space.
 
\noindent In an infinite dimensional setting, the following definitions are strictly related to the different notions of weak and strong topology.
\begin{definition}
A functional $J:\Lambda\to\bar{\mathbb{R} }$ is said to be strongly (weakly, respectively) lower semi-continuous if $J$ is lower semi-continuous when $\Lambda$ is equipped with the strong (weak, respectively) topology.
\end{definition}
\begin{definition} \label{swl}
A functional $J:\Lambda\to\bar{\mathbb{R} }$ is said to be strongly (weakly, respectively) sequentially lower semi-continuous if
\[
\liminf_{n\to +\infty} J(\lambda_n)\ge J(\lambda)
\]
for any sequence $(\lambda_n)_n\subset \Lambda$ such that $\lambda_n\to \lambda$ ($\lambda_n\rightharpoonup \lambda$, respectively).
\end{definition}
We proceed by providing some useful results.
\begin{proposition}
The following statements are equivalent:
\begin{itemize}
    \item[i)] $J:\Lambda\to\mathbb{R} $ is sequentially weakly lower semi-continuous functional;
    \item[ii)] the epigraph of $J$ is weakly sequentially closed, where, by definition, it is
    \[
    {\rm epi}(J) = \{(\lambda, t)\in {\rm dom}(J)\times\mathbb{R}  : J(\lambda)\le t\}.
    \]
\end{itemize}
\end{proposition}
\begin{remark}
As a further consequence of the preliminary Theorem \ref{noto}, we have that sequential weak lower semi-continuity and weak lower semi-continuity do not match if $\Lambda$ is infinite dimensional since weak topology is not metrizable. However, the weaker concept of sequential weak lower semi-continuity meets our needs.
\end{remark}
\begin{proposition} \label{Ccc}
Let $\cal{C}\subseteq$ $\Lambda$ be a closed and convex subset. Then, $\cal{C}$ is weakly sequentially closed, too.
\end{proposition}
Since a sequentially weakly closed set is also strongly closed, it follows that a sequentially weakly lower semi-continuous functional is also (strongly) lower semi-continuous. Instead, the converse holds under an additional assumption.
In particular, Proposition \ref{Ccc} allows us to infer the following results.
\begin{proposition} \label{help}
If $J:\Lambda\to\mathbb{R} $ is a strongly lower semi-continuous convex functional; thus $J$ is weakly sequentially lower semi-continuous, too.
\end{proposition}
\begin{proof}
Since $J$ is lower semi-continuous, thus ${\rm epi}(J)$ is closed. On the other hand, since $J$ is convex, so it is ${\rm epi}(J)$, whence Proposition \ref{Ccc} ensures that ${\rm epi}(J)$ is weakly sequentially closed, i.e., $J$ is weakly sequentially lower semi-continuous.
\end{proof}
Thus, we are able to state the main result of this section.
\begin{theorem}  \label{Hilbert}
Let $\cal{C}\subset$ $\Lambda$ be a non-empty, closed, bounded and convex subset. Let $J:\Lambda\to\mathbb{R} $ be a lower semi-continuous and convex functional. Thus $J$ achieves its minimum in $\cal{C}$, i.e., $\bar{\lambda}\in \cal{C}$ exists such that $J(\bar{\lambda}) =\displaystyle\inf_{\lambda\in \cal{C}} J(\lambda)$.
\end{theorem}
\begin{proof}
Let $m:=\displaystyle\inf_{\lambda\in \cal{C}} J(\lambda)$; hence $(\lambda_n)_n\subset \cal{C}$ exists such that
\begin{equation} \label{Jtom}
   J(\lambda_n)\to m \quad\mbox{ as } n\to +\infty.
\end{equation}
Now, our boundness assumption on $\cal{C}$ implies that, up to subsequences, $\lambda\in\mathcal{H}$ exists such that $\lambda_n\rightharpoonup \lambda$ as $n\to +\infty$. Actually, since $\cal{C}$ is a closed and convex subset of $\Lambda$, thus Proposition \ref{Ccc} applies, which guarantees that $\lambda\in \cal{C}$.\\
Finally, from \eqref{Jtom}, Proposition \ref{help} and Definition \ref{swl} we infer that $J(\bar{\lambda})\le m$, which gives the desired result. 
\end{proof}
\begin{remark}
If every closed and bounded subset in a metric space is compact, the space is said to have the Heine--Borel property. This property holds in every finite dimensional normed space but, in general, may not be true.
\end{remark}
\begin{remark}
We observe that Theorem \ref{Hilbert} still holds if the subset $\cal{C}$ is not bounded as long as we ask for an additional assumption on the functional $J$. In fact, requiring $J$ to be coercive\footnote{We say that a functional $J:H\to\mathbb{R} $ is coercive if $J(u)\to\infty$ as $\|u\|\to\infty, u\in H$.} (and if at least $\bar{\lambda}\in \cal{C}$ exists such that $J(\bar{\lambda})<+\infty$), then any minimizer of $J$ on $\cal{C}$ necessarily lies in some closed ball of radius $r>0$. In fact, since $J(\bar{\lambda})<+\infty$, any minimizer $\lambda$ of $J$ must have $J(\lambda)\le J(\bar{\lambda})$; furthermore, since $J$ is coercive, a sufficient large radius $r>0$ exists such that $J(\lambda)>J(\bar{\lambda})$ for all $\lambda\in\cal{C}$ with $\|\lambda\|> r$. Thus, any minimizer, if exists, lies in the ball $\{\lambda\in\mathcal{C}: \|\lambda\| \le r \}$.\\ In particular, Theorem \ref{Hilbert} applies to the intersection between $\cal{C}$ and a closed ball of suitable radius, since it turns to be convex if we formally require $\cal{C}$ to be closed and convex.
\end{remark}
Namely, the following result holds.
\begin{corollary}  \label{Hcor}
Let $\cal{C}\subset$ $\Lambda$ be a non-empty, closed and convex subset. Let $J:\Lambda\to\mathbb{R} $ be a lower semi-continuous, convex and coercive functional. Thus $J$ achieves its minimum, i.e., $\bar{\lambda}\in \cal{C}$ exists such that $J(\bar{\lambda}) =\displaystyle\inf_{\lambda\in \cal{C}} J(\lambda)$.
\end{corollary}
Now we introduce a couple of results which are a direct consequence of Ekeland's variational principle. 
For the sake of completeness, here we provide them with all the details (see~\cite{drusvyatskiy2019nonsmooth} for the original statements).
Let $\Lambda$ be a complete metric space and $J : \Lambda \to \mathbb{R} $ be the lower semicontinuous response function on $\Lambda$. Supposte that a point $\lambda\in\Lambda$ exists such that $J(\lambda)<+\infty$. Thus, the following results hold. 
\begin{theorem}[Perturbation Result] \label{thm:nonsmooth}
Let $J_{\lambda} : \Lambda \to \bar{\mathbb{R} }$ be a lower semicontinuous function such that the inequality
\begin{equation}  \label{fleq}
    \abs{J_{\lambda} (\gamma ) - J(\gamma)} \le \zeta(d(\gamma,\lambda)) \quad \text{holds} \quad \forall \gamma \in \Lambda,
\end{equation}
where $J_{\lambda}(\cdot)$ denote \textit{model function}\footnote{As $model$ $function$ we mean the Taylor's expansion of $J$ in $\lambda$, stopped to the first order.}, $\zeta$ is some growth function\footnote{A differentiable univariate function $\zeta: \mathbb{R} _+ \to \mathbb{R} _+$ is called a growth function if it satisfies $\zeta(0) = \zeta'(0) = 0$ and $\zeta' > 0$ on $(0, +\infty)$. If in addition,
equalities $\lim\limits_{t \to 0}\zeta'(t) =\lim\limits_{t \to 0} \zeta(t)/\zeta'(t) =0$  hold, we say that $\zeta$ is a proper growth function.}, and let $\lambda^+$ be a minimizers of $J_{\lambda}$. If $\lambda^+$ coincides with $\lambda$, then $\abs{\nabla J(\lambda)}=0$. On the other hand, if $\lambda$ and $\lambda^+$ are distinct, then a point $\hat{\lambda} \in X$ exists which satisfies
\begin{enumerate}
    \item $d(\lambda^+,\hat{\lambda}) \leq 2 \cdot \frac{\zeta(d(\lambda^+, \lambda))}{\zeta'(d(\lambda^+, \lambda))} \quad$ (point proximity)
    \item $J(\hat{\lambda}) \leq J(\lambda^+) + \zeta(d(\lambda^+, \lambda)) \quad$ (value proximity).
\end{enumerate}
\end{theorem}
\begin{proof}
By Taylor's theorem it is simple to verify that  $\abs{\nabla J_{\lambda}}(\lambda) = \abs{\nabla J}(\lambda)$. \\ Now, since $\lambda$ is a minimizer, we have $\abs{\nabla J(\lambda)}=0$ if $\lambda^+ = \lambda$. On the other hand, if $\lambda^+ \neq \lambda$, from inequality \eqref{fleq} and the definition of $\lambda^+$, it follows that
\begin{equation*}
    J(\gamma) \ge J_{\lambda}(\lambda^+) - \zeta (d(\gamma, \lambda)).
\end{equation*}
Let us define the new function
\[
G(\gamma):= J(\gamma) + \zeta(d(\gamma, \lambda)).
\]
Thus, from assumption \eqref{fleq} and inequality $\inf G \geq J_{\lambda}(\lambda^+)$ we infer that
\begin{equation*}
    G(\lambda^+) - \inf G \leq J(\lambda^+)- J_{\lambda}(\lambda^+) + \zeta(d(\lambda^+, \lambda)) \leq 2 \zeta(d(\lambda^+, \lambda)).
\end{equation*}
Whence, Theorem \ref{Eke} applies and, having $\varepsilon := 2 \zeta(d(\lambda^+, \lambda))$, for all $\rho >0$ $\lambda_{\rho}$ exists such that
\begin{equation*}
     G(\lambda_{\rho}) \leq G(\lambda^+) \quad\mbox{ and }\quad d(\lambda^+, \lambda_{\rho}) \leq \frac{\varepsilon}{\rho}.
\end{equation*}
The desired result follows simply by placing $\rho = \zeta'(d(\lambda^+, \lambda))$ with $\hat{\lambda}=\lambda_{\rho}$.
\end{proof}
An immediate consequence of Theorem \ref{thm:nonsmooth} is the following subsequence convergence result.

\begin{corollary}[Subsequence convergence to stationary points] \label{cor:sub}
Consider a sequence of points $\lambda_k$ and closed functions $J_{\lambda_k} : \Lambda \to \bar{\mathbb{R} }$ satisfying $\lambda_{k+1} = \argmin\limits_\gamma J_{\lambda_k} (\gamma)$ and $d(\lambda_{k+1}, \lambda_k) \to 0$. Moreover suppose that the inequality
\begin{equation}
    \abs{J_{\lambda_k}(\gamma) - J(\gamma)} \le \zeta(d(\lambda_k,\gamma)) \quad \text{holds} \quad \forall k\in\N \quad \text{and} \quad \gamma \in \Lambda,
\end{equation}
where $\zeta$ is a proper growth function. If $(\lambda^*, J(\lambda^*))$ is a limit point of the sequence $(\lambda_k, J(\lambda_k))$, then $\lambda^*$ is stationary for $J$.
\end{corollary}

Two interesting consequences for convergence analysis flow from there. Suppose that the models are chosen in such a way that the step-sizes $\|\lambda_{k+1} - \lambda_k\|$ tend to zero. This assumption is often enforced by ensuring that $J(\lambda_{k+1})< J(\lambda_k)$ by at least a multiple of $\|\lambda_{k+1} - \lambda_k\|^2$ (sufficient decrease condition). 
Then, assuming for simplicity that $J$ is continuous on its domain, any limit point $\lambda^*$ of the iterate sequence $\lambda_k$ will be stationary for the problem (Corollary \ref{cor:sub}).\\
Thus, by choosing an error $\varepsilon$, we can stop update \eqref{lambdaupdate} for GB algorithms in the context of bi-level HPO for penalty hyperparameter, according to the pseudo code in \ref{pcode}.

\begin{algorithm}
\caption{Pseudo-code}\label{pcode}
\begin{algorithmic}[1]
\Require Error $\varepsilon$. Some starting points $\lambda_0$, $\lambda_1$.
\Ensure Optimum $\lambda^*$ 
\While{$\|\lambda_{t} - \lambda_{t-1}\| > \varepsilon$}
    \State{Compute $\mathbf{h}(\lambda)$;}
    
    \State{update $\lambda_{t+1} = \lambda_t -\alpha \mathbf{h}_t(\lambda_t)$;}
    \State{$t+ = 1$.}
\EndWhile

\end{algorithmic}
\end{algorithm}


\section{Discussion and practical considerations}
\label{sec:discussion}
We want to emphasize that moving to infinite dimensional Hilbert spaces are not a mere abstract pretense, but it is also important in some application contexts. For example, when Support Vector Machine (SVM) are taken into consideration, a well known ``kernel trick” permits to interpret a Gaussian kernel as an inner product in a feature space. This is potentially infinite-dimensional, allowing to read the SVM classifier function as a linear function in the feature space~\cite{rossi2005classification}. Another example is provided by the quantum system possible states problem, in which the state of a free particle can be described as vectors residing in a complex separable Hilbert space~\cite{ying2016foundations}.\\
Indeed, the strength of this article lies in theory. Both the existence theorem and the stopping criterion allow us to build an approach based on solid mathematical foundations useful for future extensions and generalizations to other problems, too. For example, infinite-dimensional Covariance Descriptors (CovDs) for classification is a fertile application arena for the extensions developed here. This finds motivation in the fact that CovDs could be mapped to Reproducing Kernel Hilbert Space (RKHS) via the use of SPD-specific kernels~\cite{harandi2014bregman}.

\section{Conclusions}
\label{sec:conclusions}
In this paper, we studied the task of penalty HPO and we provided a mathematical formulation, based on Hilbert spaces, to address this issue in an unsupervised context.
\\Focusing on bi-level formulation, we showed some relaxed theoretical results both to weaken the hypotheses necessary for the existence of the solution.\\
Our approach differs from the more standard techniques in reducing the random or black box strategies giving stronger mathematical generalization suitable also when it is not possible obtaining exact minimizer.\\
We also propose to use the Ekeland's principle as a stopping criterion, which fits well in the context of GB methods.

\section*{Declarations}

\begin{itemize}
\item \textbf{Funding:} The author F. E. was funded by REFIN Project, grant number 363BB1F4, Reference project idea UNIBA027 ``Un modello numerico-matematico basato su metodologie di algebra lineare e multilineare per l'analisi di dati genomici".
The author C. S. was partially supported by PRIN project “Qualitative and quantitative aspects of nonlinear PDEs” (2017JPCAPN\_005) funded by Ministero dell’Istruzione, dell’Università e della Ricerca.
\item \textbf{Conflict of interest:} The authors have no relevant financial or non-financial interests to disclose.
\item \textbf{Data availability:} Data sharing not applicable to this article as no datasets were generated or analysed during the current study.
\end{itemize}

\bibliography{sn-bibliography}


\begin{thebibliography}{20}
\ifx \bisbn   \undefined \def \bisbn  #1{ISBN #1}\fi
\ifx \binits  \undefined \def \binits#1{#1}\fi
\ifx \bauthor  \undefined \def \bauthor#1{#1}\fi
\ifx \batitle  \undefined \def \batitle#1{#1}\fi
\ifx \bjtitle  \undefined \def \bjtitle#1{#1}\fi
\ifx \bvolume  \undefined \def \bvolume#1{\textbf{#1}}\fi
\ifx \byear  \undefined \def \byear#1{#1}\fi
\ifx \bissue  \undefined \def \bissue#1{#1}\fi
\ifx \bfpage  \undefined \def \bfpage#1{#1}\fi
\ifx \blpage  \undefined \def \blpage #1{#1}\fi
\ifx \burl  \undefined \def \burl#1{\textsf{#1}}\fi
\ifx \doiurl  \undefined \def \doiurl#1{\url{https://doi.org/#1}}\fi
\ifx \betal  \undefined \def \betal{\textit{et al.}}\fi
\ifx \binstitute  \undefined \def \binstitute#1{#1}\fi
\ifx \binstitutionaled  \undefined \def \binstitutionaled#1{#1}\fi
\ifx \bctitle  \undefined \def \bctitle#1{#1}\fi
\ifx \beditor  \undefined \def \beditor#1{#1}\fi
\ifx \bpublisher  \undefined \def \bpublisher#1{#1}\fi
\ifx \bbtitle  \undefined \def \bbtitle#1{#1}\fi
\ifx \bedition  \undefined \def \bedition#1{#1}\fi
\ifx \bseriesno  \undefined \def \bseriesno#1{#1}\fi
\ifx \blocation  \undefined \def \blocation#1{#1}\fi
\ifx \bsertitle  \undefined \def \bsertitle#1{#1}\fi
\ifx \bsnm \undefined \def \bsnm#1{#1}\fi
\ifx \bsuffix \undefined \def \bsuffix#1{#1}\fi
\ifx \bparticle \undefined \def \bparticle#1{#1}\fi
\ifx \barticle \undefined \def \barticle#1{#1}\fi
\bibcommenthead
\ifx \bconfdate \undefined \def \bconfdate #1{#1}\fi
\ifx \botherref \undefined \def \botherref #1{#1}\fi
\ifx \url \undefined \def \url#1{\textsf{#1}}\fi
\ifx \bchapter \undefined \def \bchapter#1{#1}\fi
\ifx \bbook \undefined \def \bbook#1{#1}\fi
\ifx \bcomment \undefined \def \bcomment#1{#1}\fi
\ifx \oauthor \undefined \def \oauthor#1{#1}\fi
\ifx \citeauthoryear \undefined \def \citeauthoryear#1{#1}\fi
\ifx \endbibitem  \undefined \def \endbibitem {}\fi
\ifx \bconflocation  \undefined \def \bconflocation#1{#1}\fi
\ifx \arxivurl  \undefined \def \arxivurl#1{\textsf{#1}}\fi
\csname PreBibitemsHook\endcsname

\bibitem{nostro}
\begin{bchapter}
\bauthor{\bsnm{Del~Buono}, \binits{N.}},
\bauthor{\bsnm{Esposito}, \binits{F.}},
\bauthor{\bsnm{Selicato}, \binits{L.}}:
\bctitle{Methods for hyperparameters optimization in learning approaches: An
  overview}.
In: \bbtitle{International Conference on Machine Learning, Optimization, and
  Data Science},
pp. \bfpage{100}--\blpage{112}
(\byear{2020}).
\bcomment{Springer}
\end{bchapter}
\endbibitem

\bibitem{franceschi2021unified}
\begin{botherref}
\oauthor{\bsnm{Franceschi}, \binits{L.}}:
A unified framework for gradient-based hyperparameter optimization and
  meta-learning.
PhD thesis,
UCL (University College London)
(2021)
\end{botherref}
\endbibitem

\bibitem{feurer2019hyperparameter}
\begin{bbook}
\bauthor{\bsnm{Feurer}, \binits{M.}},
\bauthor{\bsnm{Hutter}, \binits{F.}}:
\bbtitle{Hyperparameter Optimization},
pp. \bfpage{3}--\blpage{33}.
\bpublisher{Springer},
\blocation{USA}
(\byear{2019})
\end{bbook}
\endbibitem

\bibitem{bard2013practical}
\begin{bbook}
\bauthor{\bsnm{Bard}, \binits{J.F.}}:
\bbtitle{Practical Bilevel Optimization: Algorithms and Applications}
vol. \bseriesno{30}.
\bpublisher{Springer},
\blocation{Boston}
(\byear{2013})
\end{bbook}
\endbibitem

\bibitem{colson2007overview}
\begin{barticle}
\bauthor{\bsnm{Colson}, \binits{B.}},
\bauthor{\bsnm{Marcotte}, \binits{P.}},
\bauthor{\bsnm{Savard}, \binits{G.}}:
\batitle{An overview of bilevel optimization}.
\bjtitle{Annals of operations research}
\bvolume{153}(\bissue{1}),
\bfpage{235}--\blpage{256}
(\byear{2007})
\end{barticle}
\endbibitem

\bibitem{bertrand2020implicit}
\begin{bchapter}
\bauthor{\bsnm{Bertrand}, \binits{Q.}},
\bauthor{\bsnm{Klopfenstein}, \binits{Q.}},
\bauthor{\bsnm{Blondel}, \binits{M.}},
\bauthor{\bsnm{Vaiter}, \binits{S.}},
\bauthor{\bsnm{Gramfort}, \binits{A.}},
\bauthor{\bsnm{Salmon}, \binits{J.}}:
\bctitle{Implicit differentiation of lasso-type models for hyperparameter
  optimization}.
In: \bbtitle{International Conference on Machine Learning},
pp. \bfpage{810}--\blpage{821}
(\byear{2020}).
\bcomment{PMLR}
\end{bchapter}
\endbibitem

\bibitem{phdfranceschi}
\begin{botherref}
\oauthor{\bsnm{Franceschi}, \binits{L.}}:
A unified framework for gradient-based hyperparameter optimization and
  meta-learning.
PhD thesis,
University College London
(2021)
\end{botherref}
\endbibitem

\bibitem{nostro_Z}
\begin{botherref}
\oauthor{\bsnm{{Del Buono}}, \binits{N.}},
\oauthor{\bsnm{Esposito}, \binits{F.}},
\oauthor{\bsnm{Selicato}, \binits{L.}},
\oauthor{\bsnm{Zdunek}, \binits{R.}}:
Optimizing Penalization Hyperparameters in nonnegative matrix factorizations
  problems.
Preprint at \url{http://arxiv.org/abs/2203.13129}
(2022)
\end{botherref}
\endbibitem

\bibitem{franceschi2017forward}
\begin{bchapter}
\bauthor{\bsnm{Franceschi}, \binits{L.}},
\bauthor{\bsnm{Donini}, \binits{M.}},
\bauthor{\bsnm{Frasconi}, \binits{P.}},
\bauthor{\bsnm{Pontil}, \binits{M.}}:
\bctitle{Forward and reverse gradient-based hyperparameter optimization}.
In: \bbtitle{International Conference on Machine Learning},
pp. \bfpage{1165}--\blpage{1173}
(\byear{2017}).
\bcomment{PMLR}
\end{bchapter}
\endbibitem

\bibitem{vincent2018online}
\begin{botherref}
\oauthor{\bsnm{Vincent}, \binits{D.}},
\oauthor{\bsnm{Gelly}, \binits{S.}},
\oauthor{\bsnm{Le~Roux}, \binits{N.}},
\oauthor{\bsnm{Bousquet}, \binits{O.}}:
Online hyper-parameter optimization
(2018)
\end{botherref}
\endbibitem

\bibitem{fra2}
\begin{bchapter}
\bauthor{\bsnm{Franceschi}, \binits{L.}},
\bauthor{\bsnm{Frasconi}, \binits{P.}},
\bauthor{\bsnm{Salzo}, \binits{S.}},
\bauthor{\bsnm{Grazzi}, \binits{R.}},
\bauthor{\bsnm{Pontil}, \binits{M.}}:
\bctitle{Bilevel programming for hyperparameter optimization and
  meta-learning}.
In: \bbtitle{International Conference on Machine Learning},
pp. \bfpage{1568}--\blpage{1577}
(\byear{2018}).
\bcomment{PMLR}
\end{bchapter}
\endbibitem

\bibitem{Mc}
\begin{bchapter}
\bauthor{\bsnm{Maclaurin}, \binits{D.}},
\bauthor{\bsnm{Duvenaud}, \binits{D.}},
\bauthor{\bsnm{Adams}, \binits{R.}}:
\bctitle{Gradient-based hyperparameter optimization through reversible
  learning}.
In: \bbtitle{Proc. of ICML},
pp. \bfpage{2113}--\blpage{2122}
(\byear{2015})
\end{bchapter}
\endbibitem

\bibitem{Pedro}
\begin{bchapter}
\bauthor{\bsnm{Pedregosa}, \binits{F.}}:
\bctitle{Hyperparameter optimization with approximate gradient}.
In: \bbtitle{ICML},
pp. \bfpage{737}--\blpage{746}
(\byear{2016}).
\bcomment{PMLR}
\end{bchapter}
\endbibitem

\bibitem{lorraine2020optimizing}
\begin{bchapter}
\bauthor{\bsnm{Lorraine}, \binits{J.}},
\bauthor{\bsnm{Vicol}, \binits{P.}},
\bauthor{\bsnm{Duvenaud}, \binits{D.}}:
\bctitle{Optimizing millions of hyperparameters by implicit differentiation}.
In: \bbtitle{International Conference on Artificial Intelligence and
  Statistics},
pp. \bfpage{1540}--\blpage{1552}
(\byear{2020}).
\bcomment{PMLR}
\end{bchapter}
\endbibitem

\bibitem{ekeland1974variational}
\begin{barticle}
\bauthor{\bsnm{Ekeland}, \binits{I.}}:
\batitle{On the variational principle}.
\bjtitle{Journal of Mathematical Analysis and Applications}
\bvolume{47}(\bissue{2}),
\bfpage{324}--\blpage{353}
(\byear{1974})
\end{barticle}
\endbibitem

\bibitem{walter1987real}
\begin{bbook}
\bauthor{\bsnm{Rudin}, \binits{W.}}:
\bbtitle{Real and Complex Analysis},
(\byear{1987})
\end{bbook}
\endbibitem

\bibitem{drusvyatskiy2019nonsmooth}
\begin{botherref}
\oauthor{\bsnm{Drusvyatskiy}, \binits{D.}},
\oauthor{\bsnm{Ioffe}, \binits{A.D.}},
\oauthor{\bsnm{Lewis}, \binits{A.S.}}:
Nonsmooth optimization using taylor-like models: error bounds, convergence, and
  termination criteria.
Mathematical Programming,
1--27
(2019)
\end{botherref}
\endbibitem

\bibitem{rossi2005classification}
\begin{botherref}
\oauthor{\bsnm{Rossi}, \binits{F.}},
\oauthor{\bsnm{Villa}, \binits{N.}}:
Classification in hilbert spaces with support vector machines.
Proceedings of ASMDA,
635--642
(2005)
\end{botherref}
\endbibitem

\bibitem{ying2016foundations}
\begin{bbook}
\bauthor{\bsnm{Ying}, \binits{M.}}:
\bbtitle{Foundations of Quantum Programming},
(\byear{2016})
\end{bbook}
\endbibitem

\bibitem{harandi2014bregman}
\begin{bchapter}
\bauthor{\bsnm{Harandi}, \binits{M.}},
\bauthor{\bsnm{Salzmann}, \binits{M.}},
\bauthor{\bsnm{Porikli}, \binits{F.}}:
\bctitle{Bregman divergences for infinite dimensional covariance matrices}.
In: \bbtitle{Proceedings of the IEEE Conference on Computer Vision and Pattern
  Recognition},
pp. \bfpage{1003}--\blpage{1010}
(\byear{2014})
\end{bchapter}
\endbibitem

\end{thebibliography}

\end{document}